\documentclass{amsart}

\usepackage[utf8]{inputenc}
\usepackage[T1]{fontenc}
\usepackage{textcomp}
\usepackage{amsfonts}
\usepackage{amsthm}
\usepackage{amssymb}
\usepackage{pst-node}
\usepackage[all]{xy}
\usepackage{mathtools}
\usepackage{chngcntr}
\usepackage{thmtools}
\usepackage{mathrsfs}
\usepackage{tcolorbox}
\usepackage{listings}
\usepackage{imakeidx}
\usepackage{hyperref}
\usepackage{color}
\usepackage{xcolor}
\usepackage{float}
\usepackage{tikz}
\usepackage{enumerate}
\usepackage[nameinlink, capitalize, noabbrev]{cleveref}

\usetikzlibrary{shapes.geometric,decorations.markings}
\usetikzlibrary{decorations.pathreplacing}
\usetikzlibrary{fit}
\usetikzlibrary{automata}
\usetikzlibrary{positioning}
\usetikzlibrary{intersections}
\tikzset{mytext/.style={font=\small, text=black}}
\tikzset{main node/.style={circle,fill=lime!30,draw,minimum size=0.5cm,inner sep=0pt},
            }

\usepackage{pst-node}
\usepackage{pst-tree}

\def\BALL[#1](#2){\rput[t](#2){}%
        \pscircle[fillstyle=solid,fillcolor=#1!40](#2){5pt}}
        
\psset{treesep=1.3,levelsep=1.5}

\hypersetup{
    colorlinks=true,
    linkcolor=teal,
    citecolor=magenta,
    }

\newtheorem{Theorem}{Theorem}

\newtheorem{Corollary}[Theorem]{Corollary}
\newtheorem{proposition}{Proposition}[section]
\newtheorem{lemma}[proposition]{Lemma}
\newtheorem{corollary}[proposition]{Corollary}
\newtheorem{theorem}[proposition]{Theorem}

\newtheorem{Question}[proposition]{Question}
\theoremstyle{definition}

\newtheorem{definition}[proposition]{Definition}

\def \<#1>{{\left\langle{#1}\right\rangle}}
\def\abs#1{\left\vert{#1}\right\vert}
\def\set#1{{\def\st{\;:\;}\left\{#1\right\}}}

\def\Z-{\overline{\mathbb Z}}
\def\Zl-{\overline{\mathbb Z}_\ell}

\def\Q-{\overline{\mathbb Q}}
\def\Ql-{\overline{{\mathbb Q}_\ell}}
\def\K-{\overline{K}}

\def\Fl-{\overline{{\mathbb F}_\ell}}
\def\FF{\mathbb F}

\def\PP{\mathbb P}

\DeclareMathOperator{\St}{St}

\DeclareMathOperator{\Aut}{Aut}

\numberwithin{equation}{section}

\title{Random subgroups of branch groups}
\author{Jorge Fariña-Asategui and Santiago Radi}
\address{Jorge Fariña-Asategui: Centre for Mathematical Sciences, Lund University, 223 62 Lund, Sweden -- Department of Mathematics, University of the Basque Country UPV/EHU, 48080 Bilbao, Spain}
\email{jorge.farina\_asategui@math.lu.se}
\address{Santiago Radi: Department of Mathematics, Texas A\&M University, 77843 College Station, U.S.A.
}
\email{santiradi@tamu.edu}
\keywords{Torsion elements, branch groups, Haar measure, pro-$p$ groups}
\subjclass[2020]{Primary: 20E08, 20E18; Secondary: 28C10}
\thanks{The first author is supported by the Spanish Government, grant PID2020-117281GB-I00, partly with FEDER funds. The first author also acknowledges support from the Walter Gyllenberg Foundation from the Royal Physiographic Society of Lund. The second author is supported by Grigorchuk's Simons Foundation Grant MP-TSM-00002045 and the department of Mathematics of Texas A\&M University.}

\begin{document}

\begin{abstract} We show that independent Haar-random elements in a super strongly fractal branch profinite group generate a free subgroup acting freely on the boundary of the tree. This improves a previous result of Abért (2005) for weakly branch profinite groups, where independent random elements were shown to generate free subgroups acting only almost freely on the boundary. Our result also generalizes the analogous result of Abért and Virág (2005) for iterated wreath products.

\end{abstract}
\maketitle

\section{introduction}
\label{section: introduction}

Random subgroups of groups are subgroups obtained by a probabilistic construction. One is usually interested in the following questions: How does a typical subgroup look like? What properties does a typical subgroup have?

Erd\H{o}s and Turán started the field of statistical group theory in a series of papers beginning in 1965. They established many important properties of (uniform) random elements of symmetric groups. For example, they computed the asymptotic distribution of the order of random elements in $\mathrm{Sym}(n)$ as $n\to\infty$ \cite{ErdosTuran}. In 1969, Dixon proved a seminal result in random generation of finite groups \cite{Dixon}: two random elements of the alternating group $A_n$ generate $A_n$ with probability tending to~1 as $n\to\infty$. Dixon's result has been successfully extended to all non-abelian finite simple groups \cite{Kantor, LS}. This line of investigation has been followed until today, resulting in many remarkable results; cf. \cite{AbertVirag, BurnesLiebeckShalev, JaikinPyber, LiebeckShalev1}.

Probabilistic questions in finite groups generalize easily to infinite profinite groups. Indeed, profinite groups may be studied via their finite quotients and they admit a well-behaved probability space structure with respect to a unique normalized Haar measure. A first natural probabilistic question to consider in infinite profinite groups is what kind of algebraic properties the subgroups generated by independent Haar-random elements satisfy. Furthermore, profinite groups often arise as groups acting on some set, so one may also wonder what kind of properties the action of such random subgroups has. A remarkable result in this vein was obtained by Abért in \cite{Abert}.

We say that an action of a group $G$ on a set $X$ is \textit{separating} if the pointwise stabilizer of any finite subset $S\subset X$ does not fix any point in the complement of~$S$. The action of a compact group $G$ on $X$ is \textit{topological} if every point stabilizer is closed in $G$ and of Haar measure zero. If the action of the compact group $G$ on $X$ is transitive, the set $X$ can be identified with the coset space $G/\mathrm{Stab}(x)$ for any $x \in X$, and $X$ inherits a $G$-invariant probability measure $\lambda$ from the Haar measure in $G$. In this case, the group $G$ is said to act \textit{almost freely} on $X$ if there exists a conull subset $Y\subseteq X$ (with respect to $\lambda$) such that  $Y$ is $G$-invariant and $G$ acts freely on $Y$. The aforementioned result of Abért reads as follows: 

\begin{theorem}[{see {\cite[Theorem 1.3]{Abert}}}]
\label{theorem: abert}
Let $G$ be a compact group acting topologically on a measurable space $X$ via a separating action. Then, for any $k \ge 1$, $k$ independent random elements in $G$ generate a free subgroup $H$ of rank $k$ almost surely. Moreover, if the $G$-action on $X$ is transitive, then almost surely $H$ acts almost freely on $X$.
\end{theorem}

For many applications, the action of random subgroups needs to be (almost surely) free rather than almost free. However, the main tool used by Abért in his proof of \cref{theorem: abert} is Fubini's theorem, so one can only study the action of random elements up to a set of measure zero in $X$. Therefore, a different method is needed.

In the case of iterated wreath products acting on a regular rooted tree $T$, Abért and Virág introduced a new approach in \cite{AbertVirag} to study random elements of iterated wreath products based on stochastic processes. They defined the orbit tree of a random element (and of a random subgroup) and showed that it has the same distribution as a Galton-Watson process. This allows one to use the theory of Galton-Watson processes and branching random walks to deduce certain properties about random subgroups of iterated wreath products. In particular, they showed that the action on the boundary of the tree of a random subgroup of a level-transitive iterated wreath product is free almost surely; see \cite[Corollary~4.2]{AbertVirag}.

The downside of the approach in \cite{AbertVirag} is that it is very specific to iterated wreath products: the orbit tree of a random element in $G$ is Galton-Watson if and only if $G$ is an iterated wreath product. However, even if the orbit tree of a random element is not Galton-Watson, one might still hope to obtain probabilistic independence of sections which are far enough from each other. In this paper, we show that a sufficient condition for this probabilistic independence is that the group is super strongly fractal and branch; see \cref{lemma: sections are independent}. This should not come as a surprise, as super strongly fractal groups give rise to strongly mixing measure-preserving dynamical systems, by considering the natural action of the tree on the group via sections \cite{JorgeCyclicity}. Furthermore, the first author has shown in \cite{JorgeMarkov} that the dynamical systems associated to fractal branch groups are Markov-processes over a free semigroup; see \cite{BowenMarkov} for more details on Markov processes over free groups and semigroups.

Our main result in this paper generalizes the result of Abért and Virág in \cite[Corollary 4.2]{AbertVirag} and shows that for certain branch groups the action on $\partial T$ of the random subgroups in \cref{theorem: abert} is indeed free. Note that we call a subgroup generated by $k$ independent Haar-random elements a \textit{$k$-generated random subgroup}.

\begin{Theorem}
\label{Theorem: main result}
Let $G\le \mathrm{Aut}(T)$ be a super strongly fractal branch profinite group. Then almost surely, for every $k\ge 1$, a $k$-generated random subgroup of $G$ is a free group of rank $k$ which acts freely on $\partial T$.
\end{Theorem}

As level-transitive iterated wreath products are super strongly fractal and branch, \cref{Theorem: main result} widely generalizes \cite[Corollary 4.2]{AbertVirag}, and gives a more detailed description of the action of the random subgroups in \cref{theorem: abert} for a more restrictive family of groups. \cref{Theorem: main result} applies to the closure of the most studied examples of branch groups such as the first Grigorchuk group \cite{GrigorchukBurnside,JoneFractal}, the periodic Grigorchuk-Gupta-Sidki groups \cite{GGSHausdorff, JoneFractal}, the iterated monodromy group of the complex polynomial $z^2+i$ \cite{IMG, RadiFiniteType2025}, or the Hanoi towers group in $3$ pegs \cite{GrigNekraSunic2006, RadiFiniteType2025}.

As the examples presented in \cref{section: examples} show, freeness of the action of random subgroups need not hold if one drops either of the assumptions in \cref{Theorem: main result}.

We remark that even if the approach to prove \cref{Theorem: main result} is based on probabilistic independence of sections as in \cite[Corollary 4.2]{AbertVirag}, the proof of this probabilistic independence requires completely different tools compared to \cite[Corollary 4.2]{AbertVirag}. In fact, our proof relies heavily on the techniques developed by the first author in~\cite{JorgeCyclicity}.

It is worth mentioning that \cref{Theorem: main result} opens the door to using similar ideas as in \cite{AbertVirag} to study random elements of a wide class of branch groups. This has further potential applications to random matrix theory, via the study of the natural matrix representations of the congruence quotients of super strongly fractal branch groups \cite{Evans}.

We now show two applications of \cref{Theorem: main result}. Recall that a group is said to be \textit{amenable} if it admits a left-invariant mean. In \cite{GrigorchukBurnside}, Grigorchuk constructed the first example of an amenable but not elementary amenable group. This group, nowadays known as the first Grigorchuk group, was also the first example of a group of intermediate growth. In general, any group of subexponential growth is amenable as an application of Kesten's criterion \cite{Kesten}. On the other hand, it is well known that non-abelian free groups are not amenable. By \cref{theorem: abert}, this implies that if $G$ is a compact topological group satisfying the assumptions in \cref{theorem: abert}, then a $k$-random subgroup, with $k \geq 2$, is not amenable almost surely.

Even if a group $G$ is not amenable, it may admit amenable actions, i.e. there may exist a $G$-set $X$ which admits a $G$-invariant mean. Quite remarkably, non-abelian free groups do admit faithful and transitive amenable actions by a posthumous result of Douwen; see \cite{Dou}. Further examples of non-amenable groups admitting a faithful and transitive amenable action were obtained by Glasner and Monod in \cite{GlasnerMonod} as free products of groups not having property~(F). Further faithful and transitive amenable actions of non-abelian free groups were constructed by Grigorchuk and Nekrashevych in \cite{GrigNekra} using groups acting on rooted trees.

In view of \cref{theorem: abert} and the aforementioned examples of amenable actions of non-abelian free groups, one may wonder whether the action on the boundary of the random subgroups in \cref{theorem: abert} of the tree can be amenable or not. Since a group admitting a free amenable action is amenable and a non-abelian free group is not amenable, \cref{Theorem: main result} yields a negative answer for super strongly fractal branch profinite groups:

\begin{Corollary}
    Let $G\le \mathrm{Aut}(T)$ be a super strongly fractal branch profinite group. Then for every $k\ge 2$, the action of a $k$-generated random subgroup of $G$ on $\partial T$ is not amenable almost surely. 
\end{Corollary}

A second application is to arithmetic dynamics and number theory. Free actions of random elements of the Galois groups associated to the iterates of a rational function $f$ are often enough to obtain results on different prime density questions concerning the iterates of the rational function $f$; see the surveys \cite{SurveyAD, JonesGalois} and \cite{BridyJones2022, JorgeSantiFPP, Jones2007, Jones15, Jones2008, Juul2016, Odoni}. Translated to the language of groups acting on rooted trees, if $G \leq \mathrm{Aut}(T)$ is a closed subgroup, one is interested in whether a random element of $G$ acts freely on the boundary $\partial T$ almost surely. More generally, if we denote by $\mu$ the Haar measure on $G$, we may define the \textit{fixed-point proportion} of $G$ as
$$\mathrm{FPP}(G) := \mu(\set{g \in G: \text{$g$ fixes an element in $\partial T$}}).$$
One is usually interested in whether $\mathrm{FPP}(G)=0$.

A recent result of the second author in \cite{Santi} shows that certain branch groups arising as Galois groups of the iterates of a polynomial can have non-zero fixed-point proportion; see the second example in \cref{section: examples}. In the same paper, it is also asked whether a $k$-random subgroup of a group with null fixed-point proportion must also have null fixed-point proportion almost surely; see \cite[Question 7.3]{Santi}. We answer this question strongly in the negative in \cref{section: examples} by providing  a locally finite super strongly fractal group whose finitely generated random subgroups have always non-zero fixed-point proportion. This example shows that a major obstruction for finitely generated random subgroups to have null fixed-point proportion is the group being locally finite. This leads to the following reformulation of \cite[Question 7.3]{Santi}:

\begin{Question}
\label{question: FPP random subgroups}
Let  $G \leq \mathrm{Aut}(T)$ be a non-locally finite closed subgroup with null fixed-point proportion. Let $H$ be a $k$-random subgroup for some $k\ge 1$. Is it $$\mathrm{FPP}(H)=0$$
almost surely?
\end{Question}

In view of this, we restrict ourselves to topologically finitely generated infinite  profinite groups. The following result suggests a positive answer to \cref{question: FPP random subgroups} in this case:

\begin{Theorem}
\label{Theorem: FPP of random subgroups}
Let $G \le \mathrm{Aut}(T)$ be a topologically finitely generated group such that $$\mathrm{FPP}(G)=0.$$
Then, a $k$-generated random subgroup of $G$ has null fixed-point proportion with probability tending to 1 as $k \to \infty$.
\end{Theorem}

The proof of \cref{Theorem: FPP of random subgroups} is based on a result of Borovik, Pyber and Shalev in \cite{MaximalGrowth} related to positive finite generation of profinite groups. If $G$ is a topologically finitely generated closed subgroup of $W_p$, where $W_p$ is the iterated wreath product generated by a $p$-cycle with $p$ prime, then $G$ is pro-$p$ and thus a version of \cref{Theorem: FPP of random subgroups} may be proved by elementary means:

\begin{Corollary}
\label{Corollary: FPP of random subgroups}
Let $G \le W_p$ be a topologically $d$-generated pro-$p$ group such that
$$\mathrm{FPP}(G)=0.$$ 
Then, if $k \geq d$, a $k$-generated random subgroup of $G$ has null fixed-point proportion with probability at least 
$$\prod_{j = 0}^{d-1} \left(1 - \frac{1}{p^{k-j}} \right).$$
In particular this probability tends to 1 as $k \to \infty$.
\end{Corollary}

In view of \cref{Theorem: main result}, we believe that random subgroups of super strongly fractal branch profinite groups have null fixed-point proportion almost surely, but the tools developed in this paper are not enough to prove this assertion.

\subsection*{\textit{\textmd{Organization}}} In \cref{section: Preliminaries} we give the necessary background needed in the subsequent sections. In \cref{section: random subgroups}, we prove the key lemma on probabilistic independence of sections and use it to prove \cref{Theorem: main result}. We also provide a proof of \cref{Theorem: FPP of random subgroups} in \cref{section: random subgroups}. Finally, in \cref{section: examples} we provide examples showing that none of the conditions in \cref{Theorem: main result} may be dropped.

\subsection*{\textit{\textmd{Notation}}} Groups will be assumed to act on the tree on the right so composition will be written from left to right. We shall write $v\cdot g$ for the action of a group element $g$ on a vertex $v$ of the tree. Finally, we denote by $\# S$ the cardinality of a finite set $S$.

\section{Preliminaries}
\label{section: Preliminaries}

\subsection{Regular rooted trees}
The \textit{$d$-regular rooted tree} $T$ is the infinite rooted tree with root~$\emptyset$, where every vertex has exactly $d$ descendants. The $d$-regular rooted tree $T$ may be identified with the free monoid on the set $X=\{1,\dotsc, d\}$. The set of vertices at a distance exactly $n \geq 1$ from the root forms the \textit{$n$-th level of $T$}, which will be denoted~$\mathcal{L}_n$. The vertices whose distance is at most $n$ from the root form the \textit{$n$-th truncated tree}~$T^n$. We define the \textit{boundary} $\partial T$ of the tree $T$ as the set of infinite words in $X$. For any vertex $v$ in $T$, the subtree rooted at $v$, which is isomorphic as a tree to $T$, is denoted~$T_v$. 

The following notion will play a central role in the present paper:

\begin{definition}[$m$-cousins]
We say that two distinct vertices $v,w\in \mathcal{L}_n$ are \textit{$m$-cousins} if there exists $1\le j\le m$ and $u \in \mathcal{L}_{n-j}$ such that both $v$ and $w$ are descendants of $u$. This is equivalent to saying that the vertices $v$ and $w$ are at distance at most $2m$ in the tree. 
\end{definition}

\begin{figure}[H]

        \centering
 \includegraphics{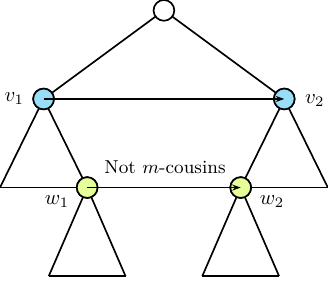}
        \caption{An illustration of vertices which are not $m$-cousins: $w_1$ and $w_2$ are at least $m$ levels below $v_1$ and $v_2$ respectively.}
        \label{figure: not m cousins}
\end{figure}

\subsection{Groups acting on regular rooted trees}

The group of automorphisms of~$T$, denoted $\Aut(T)$, is the group of graph automorphisms of $T$, i.e. functions fixing the root and preserving adjacency. In particular $\Aut(T)$ acts by permuting vertices at the same level of $T$.

The following observation will be useful to show that certain vertices are not $m$-cousins:

\begin{lemma}
\label{lemma: moving a vertex and its descendant}
    Let $g\in  \mathrm{Aut}(T)$. If $v\in \mathcal{L}_n$ is moved by $g$, then for any $w\in \mathcal{L}_m$ we have that $vw$ and $(vw)\cdot g$ are not $m$-cousins.
\end{lemma}
\begin{proof}
As $v \neq v\cdot g$, we have $d_T(v, v\cdot g) \ge 1$, where $d_T(\cdot,\cdot)$ denotes the usual graph distance in $T$. As $T$ is a tree, for every pair of vertices in $T$, there is a unique path from one to the other. Hence, we get
\begin{align*}
d_T(vw,(vw)\cdot g) &= d_T(vw,v) + d_T(v,v\cdot g) + d_T(v\cdot g, (vw)\cdot g) \\
&= 2m + d_T(v,v\cdot g) \ge  2m+1.\qedhere
\end{align*}
\end{proof}

Given a vertex $v\in T$, we write $\mathrm{st}(v)$ for the stabilizer of the vertex $v$. Furthermore, given $n \ge 1$, we define $\mathrm{St}(n)$, the \textit{stabilizer of level $n$}, as the subgroup of $\mathrm{Aut}(T)$ consisting of those automorphisms fixing level $n$ pointwisely. Each level-stabilizer $\mathrm{St}(n)$ is a normal subgroup of finite index in $\mathrm{Aut}(T)$. The filtration $\{\St(n)\}_{n\ge 1}$ induces a metrizable topology in $\mathrm{Aut}(T)$ called the \textit{congruence topology}. The group $\Aut(T)$ is a countably based profinite group with respect to this topology.

For each $n \ge 1$, we define the natural projection
$$\pi_n: \Aut(T) \rightarrow \Aut(T)/\St(n).$$
Note that if $g \in \Aut(T)$, then $\pi_n(g)$ represents the action of $g$ on the first $n$ levels of the tree. Therefore $\pi_n$ induces an isomorphism $\Aut(T)/\St_G(n)\cong\Aut(T^n)$. 

Let $v\in T$ and $g \in \Aut(T)$, the \textit{section of $g$ at $v$} is the unique automorphism $g|_v\in \mathrm{Aut}(T_v)$ such that
\begin{align*}
(vw)\cdot g=(v\cdot g)( w\cdot g|_v),
\end{align*}
for all $w \in T_v$. For all $g,h\in \mathrm{Aut}(T)$ and $w\in T_v$ we have
\begin{align}
g|_{vw} = (g|_v)|_w\quad \text{and}\quad gh|_v=g|_v\cdot h|_{v\cdot g}.
\label{equation: section two vertices}
\end{align}
By the second equality in \cref{equation: section two vertices}, it is easy to show inductively that
\begin{align}
\label{align: property of sections}
(g_1g_2\dotsb g_n)|_v=g_1|_v\cdot g_2|_{v\cdot {g_1}}\dotsb g_n|_{v\cdot {g_1g_2\dotsb g_{n-1}}}.
\end{align}

Under the identification $T_v=T$, we get $g|_v\in \mathrm{Aut}(T)$. In other words, we may define a map
\begin{align*}
\begin{split}
\varphi_v:\Aut(T) &\to \Aut(T)\\ 
g &\mapsto g|_v.
\end{split}
\end{align*}
The map $\varphi_v$ restricts to a group homomorphism on $\mathrm{st}(v)$.

More generally, for any $g\in G$, $n\ge 1$ and $v\in T $ we define the \textit{section of $g$ at $v$ of depth $n$}, as the unique automorphism $g|_v^n\in \mathrm{Aut}(T^n_v)$ such that \begin{align*}
(vw)\cdot g=(v\cdot g)(w\cdot g|_v^n),
\end{align*}
for all $w \in T_v^n$. For $n=1$ we call $g|_v^1$ the \textit{label} of $g$ at $v$. Furthermore, for any $n,m \ge 1$, and any subset of vertices $V \subseteq \mathcal{L}_n$, we further define the map $\varphi_V^{n,m}: \pi_{n+m}(\mathrm{St}(n)) \to \mathrm{Aut}(T^m)^{\# V}$ via
$$g|_\emptyset^{n+m} \mapsto \prod_{v \in V} g|_v^m.$$

Let us fix a subgroup $G \le \Aut(T)$. We say that a group $G \leq \Aut(T)$ is \textit{level-transitive} if the action of $G$ on each level $n\ge 1$ is transitive. We define vertex stabilizers and level stabilizers by restricting the ones defined for $\Aut(T)$, i.e. $\mathrm{st}_G(v) := \mathrm{st}(v) \cap G$ and $\mathrm{St}_G(n) := \mathrm{St}(n) \cap G$ for $v\in T$ and $n \ge 1$ respectively. Given $v\in T$, we define the \textit{rigid vertex stabilizer} $\mathrm{rist}_G(v)$ as the subgroup of $G$ consisting of all the automorphisms in $G$ which fix every vertex not in $T_v$. The \textit{$n$-th rigid level stabilizer} $\mathrm{Rist}_G(n)$ is the group generated by all the rigid vertex stabilizers $\mathrm{rist}_G(v)$ with $v\in \mathcal{L}_n$. Clearly $\mathrm{Rist}_G(n)$ is a normal subgroup of $G$, as $\mathrm{rist}_G(v)^g=\mathrm{rist}_G(v\cdot g)$ for any $v\in T$ and any $g\in G$. We say that $G$ is a \textit{branch group} if $G$ is level-transitive and $\mathrm{Rist}_G(n)$ is of finite index in $G$ for all $n \ge 1$. 

A group $G \leq \Aut(T)$ is \textit{self-similar} if for every $v\in T$ and $g \in G$, we have $g|_v \in G$ (under the identification $T_v=T$). Note that if $G$ is a self-similar group, then the image of the map $\varphi_v$ is contained in $G$ for every $v\in T$. A group $G$ is called \textit{fractal} if $G$ is self-similar, level-transitive and for every $v\in T$, we have $\varphi_v(\mathrm{st}_G(v)) = G$. A group $G$ is called \textit{super strongly fractal} if $G$ is self-similar, level-transitive and for all $n \ge 1$ and $v\in \mathcal{L}_n$, we have $\varphi_v(\mathrm{St}_G(n)) = G$.

Clearly, super strongly fractal implies fractal, but the converse implication is not true; see \cite[Section 3]{JoneFractal}.

A subgroup $K\le \mathrm{Aut}(T)$ is said to be \textit{branching} if $\psi(K)\ge K\times\dotsb\times K$. A self-similar, level-transitive group $G$ is \textit{regular branch} if it contains a finite-index branching subgroup $K\le G$. In particular, regular branch groups are branch.

\subsection{Groups of finite type}

Let $T$ be a $d$-regular rooted tree. For every $D\ge 1$ and $\mathcal{P}\le \Aut(T^D)$, we define the corresponding \textit{group of finite type} $G_\mathcal{P}$ as 
\begin{equation*}
G_\mathcal{P} := \set{g \in \Aut(T): g|_v^D \in \mathcal{P} \text{ for all $v \in T$}}.
\end{equation*}

The case $D = 1$ corresponds precisely to iterated wreath products. 

The following result relates groups of finite type and closed regular branch groups of $\Aut(T)$:

\begin{theorem}[{{see \cite[Theorem 3]{SunicHausdorff} and \cite[Proposition 7.5]{GrigorchukFinite}}}]
    \label{theorem: equivalence finite type}
    Let $G\le \mathrm{Aut}(T)$ be a closed subgroup. Then, the following are equivalent:
    \begin{enumerate}[\normalfont(i)]
        \item $G$ is of finite type of depth $D$;
        \item $G$ is regular branch over $\mathrm{St}_G(D-1)$.
    \end{enumerate}
\end{theorem}

In the case of fractal groups, we have the following equivalence of branchness given by the first author in \cite{JorgeSpectra}:

\begin{theorem}[{see {\cite[Theorem 3.7]{JorgeSpectra}}}]
\label{theorem: equivalence of branch properties}
For fractal closed subgroups of $\Aut(T)$, the notions of finite type, regular branch and branch are all equivalent.
\end{theorem}

\subsection{Probability on groups acting on the $d$-regular rooted tree}
\label{subsection: Probability on groups acting on dary trees}

Any closed subgroup $G\le \mathrm{Aut}(T)$ is a profinite group, so it admits a unique normalized Haar measure $\mu$. The Haar measure $\mu$ is supported on the Borel $\sigma$-algebra $\mathcal{B}_G$, which is generated by cone sets, where the \textit{cone set} $C_A$ corresponding to $A\subseteq \pi_n(G)$ is defined as 
$$C_A := \pi_n^{-1}(A) = \set{g \in G: \pi_n(g) = a \text{ for }a\in A}.$$
Moreover 
\begin{align*}
\mu(C_A) =  \frac{\# A}{|\pi_n(G)|}.
\end{align*}

Let $(\Omega, \PP, \mathcal{F})$ be a probability space and $\set{X_i}_{i \in I}$ a finite family of random variables, where $X_i: \Omega \rightarrow \Omega$ for all $i \in I$. We say that $\set{X_i}_{i \in I}$ is a collection of \textit{independent} random variables if for any collection of measurable sets $\set{A_i}_{i \in I} \subseteq \mathcal{F}$, we have 
\begin{align}
\PP \left(\bigcap_{i \in I} (X_i \in A_i) \right) = \prod_{i \in I} \PP(X_i \in A_i).
\label{equation: independence random variables}
\end{align}

It is well-known that in order to prove independence of $\set{X_i}_{i \in I}$, it is enough to check \cref{equation: independence random variables} in a generating set of the $\sigma$-algebra $\mathcal{F}$.

In this article, we will be interested in the probability space $(G, \mu, \mathcal{B}_G)$ associated to a closed subgroup $G\le \mathrm{Aut}(T)$. The following property, proved by the first author in \cite{JorgeCyclicity}, will be key in the proof of \cref{Theorem: main result}:

\begin{theorem}[{see {\cite[Theorem A]{JorgeCyclicity}}}]
Let $G \leq \Aut(T)$ be a closed fractal subgroup. Then, for every $v$, the section map $\varphi_v:G\to G$ is measuring-preserving.
\label{theorem: fractal measure preserving}
\end{theorem}

If $k \geq 1$, then the unique Haar measure in $G^k$ is no more than the product measure $\mu^k$, i.e.
$$\mu^k(A_1 \times \cdots \times A_k) = \prod_{i = 1}^k \mu(A_i).$$

As a corollary to \cref{theorem: abert}, Abért concludes (in the more general context of weakly branch groups) the following result:

\begin{corollary}[{see {\cite[Corollary 1.4]{Abert}}}]
\label{theorem: abert random free subgroup branch group}
Let $G \leq \Aut(T)$ be a closed branch group. Then, for any $k \ge 1$, $k$ independent Haar-random elements in $G$ generate a free subgroup $H$ of rank $k$ almost surely.
\end{corollary}

\subsection{Fixed-point proportion}

Given a group $G \leq \Aut(T)$, its \textit{fixed-point proportion} is defined via 
\begin{equation*}
\mathrm{FPP}(G) = \lim_{n \rightarrow \infty} \frac{\# \set{g \in \pi_n(G): \text{$g$ fixes a vertex in $\mathcal{L}_n$}}}{\abs{\pi_n(G)}}.
\end{equation*}

Clearly $\mathrm{FPP}(G) = \mathrm{FPP}(\overline{G})$ for the closure $\overline{G}$ of $G$ in $\Aut(T)$, as the definition of the fixed-point proportion only depends on the quotients with the level stabilizers.

By \cite[Lemma 16.2]{Fried1986}, if $S$ is a measurable subset of~$\overline{G}$, then 
\begin{equation*}
\mu(S) = \lim_{n \rightarrow \infty} \frac{\#\pi_n(S)}{\abs{\pi_n(G)}},
\label{equation: Haar measure limit}  
\end{equation*}
where $\mu$ is the unique Haar measure in $\overline{G}$. In particular, if we write
\begin{align*}
F :&= \set{g \in \overline{G}: g \text{ fixes a vertex in $\mathcal{L}_n$ for all $n\ge 1$}}\\
&= \mu(\set{g \in \overline{G}: \text{$g$ fixes an end in $\partial T$}}),
\end{align*}
we obtain 
\begin{equation*}
\mathrm{FPP}(\overline{G}) = \mu(F).
\label{ec_FPP_setS}
\end{equation*}

The main tool to compute the fixed-point proportion of self-similar groups is the following result proved by the authors in \cite{JorgeSantiFPP}:

\begin{theorem}[{see {\cite[Theorem 1]{JorgeSantiFPP}}}]
\label{theorem: ssf FPP 0}
Let $G\le \Aut(T)$ be a super strongly fractal group. Then 
$$\mathrm{FPP}(G)=0.$$
\end{theorem}

\section{Random subgroups}
\label{section: random subgroups}

In this section, we develop the probabilistic tools needed for the proof of \cref{Theorem: main result}. We conclude the section by proving \textcolor{teal}{Theorems} \ref{Theorem: main result} and \ref{Theorem: FPP of random subgroups}.

In view of \cref{theorem: equivalence of branch properties}, we shall assume in the following that a super strongly fractal branch profinite group is a group of finite type of depth $D$ for some $D\ge 1$.

\subsection{Probabilistic independence of sections}

First, we show that the sections of a random element at far enough vertices are independent Haar-random elements:

\begin{lemma}
\label{lemma: sections are independent}
Let $G\le \mathrm{Aut}(T)$ be a super strongly fractal group of finite type of depth~$D$. For $n\ge D$, let $V \subseteq \mathcal{L}_n$ be a set of vertices such that they are pairwise not $(D-1)$-cousins. Then, if $g$ is a Haar-random element of $G$, the sections $\{g|_v\}_{v\in V}$ form a collection of independent Haar-random elements of $G$.
\end{lemma}
\begin{proof}
Since the Borel algebra of $G$ is generated by the cone sets, it is enough to prove that $$\mu(\bigcap_{v\in V}\varphi_v^{-1}(C_{A_v})) = \prod_{v\in V}\mu(\varphi_v^{-1}(C_{A_v})),$$
for any $m\ge 1$ and any $A_v \subseteq \pi_m(G)$ for each $v\in V$.

On the one hand, by \cref{theorem: fractal measure preserving}, taking sections preserves the Haar measure:
\begin{align}
\label{align: rhs prob ind sections}
\prod_{v\in V}\mu(\varphi_v^{-1}(C_{A_v})) = \prod_{v\in V}\mu(C_{A_v}) = \prod_{v\in V} \frac{\# A_v}{\abs{\pi_m(G)}}.
\end{align}

On the other hand, we have 
\begin{align*}
\mu(\bigcap_{v\in V}\varphi_v^{-1}(C_{A_v})) &= \frac{\# \set{h \in \pi_{n+m}(G): h|_v^m \in A_v \text{ for all } v\in V}}{\abs{\pi_{n+m}(G)}} \\
&=\sum_{(b,a_v)\in \pi_n(G)\times \prod_{v\in V}A_v}\frac{\# \set{h \in \pi_{n+m}(G): h|_\emptyset^n = b, h|_v^m=a_v~\forall v\in V }}{\abs{\pi_{n+m}(G)}}.
\end{align*}

First, we show that the map $\varphi_V^{n,m}$ has image $\prod_{v\in V}\pi_m(G)$. Clearly, the image of $\varphi_V^{n,m}$ is contained in $\prod_{v\in V}\pi_m(G)$ as $G$ is self-similar, so let us prove that each element in $\prod_{v\in V}\pi_m(G)$ may be realized in the image  of $\varphi_V^{n,m}$.  Let $(g_v)_{v\in V}\in \prod_{v\in V}\pi_m(G)$. For each $v\in V$, let $v_1 \in \mathcal{L}_{n-(D-1)}$ be the unique vertex $D-1$ levels above $v$. As the vertices in $V$ are pairwise not $(D-1)$-cousins, we have $v_1= \widetilde{v}_1$ if and only if $v=\widetilde{v}$ for $v,\widetilde{v}\in V$. Let us write each $v = v_1 v_2$, where $v_2\in T^{[D-1]}$.  

As $G$ is super strongly fractal, for each $v\in V$ there exists $s_v \in \St_G(D-1)$ such that 
$$s_v|_{v_2} = g_v.$$
Now, as $G$ is regular branch over $\mathrm{St}_G(D-1)$, we have
$$\psi_{n-D+1}(\mathrm{St}_G(n))=\mathrm{St}_G(D-1)\times\overset{d^{n-D+1}}{\dotsb}\times \mathrm{St}_G(D-1).$$
Thus, there exists $s \in \St_G(n)$ such that, for every $v\in V$ we have
$$s|_{v_1} = s_v.$$
Hence 
$$\varphi_V^{n,m}(s)=(s|_v)_{v\in V}=(s_v|_{v_2})_{v\in V}=(g_v)_{v\in V}.$$

Now we show that for any tuple $(b,a_v)\in \pi_n(G)\times \prod_{v\in V}A_v$, we can find an element $h \in \pi_{n+m}(G)$ such that 
\begin{align}
\label{align: sections of h}
h|_\emptyset^n = b\quad \text{and}\quad h|_v^m=a_v \quad \text{for all }v\in V.
\end{align}

Indeed, let $g \in \pi_{n+m}(G)$ such that $\pi_n(g) = b$. As $\prod_{v\in V}\pi_m(G)$ is in the image of~$\varphi_V^{n,m}$, there exists $s \in \pi_{n+m}(\St_G(n))$ such that 
$$s|_v^m = a_v (g|_v^m)^{-1}\quad \text{for all } v \in V.$$
Then, setting $h:= sg$, we have 
$$\pi_n(h)=\pi_n(g)=b,$$
and for each $v\in V$ we have
$$h|_v^m=(sg)|_v^m = s|_v^m \cdot g|_v^m = a_v (g|_v^m)^{-1} g|_v^m = a_v$$ as wanted.

Now, if both $h,\widetilde{h}$ satisfy the conditions in \cref{align: sections of h}, then $h^{-1}\widetilde{h}$ satisfies both $\pi_n(h^{-1}\widetilde{h})=1$ and
$$(h^{-1}\widetilde{h})|_v^m =1$$
for every $v\in V$. In other words, we have the equality of cosets $\ker \varphi_V^{n,m} h=\ker \varphi_V^{n,m} \widetilde{h}$. Conversely, if $k\in \ker \varphi_V^{n,m}$ and $h$ satisfies the conditions in \cref{align: sections of h}, then $kh$ also satisfies the conditions in \cref{align: sections of h}:
$$\pi_n(kh)=\pi_n(k)\cdot  \pi_n(h)=\pi_n(h)=b$$
and 
$$(kh)|_v^m=k|_v^m\cdot h|_v^m=a_v$$
for each $v\in V$. Therefore
\begin{align}
\label{equation: horizontal vertical independence}
\mu(\bigcap_{v\in V}\varphi_v^{-1}(C_{A_v}))=\frac{\abs{\pi_n(G)} \cdot \prod_{v\in V} \# A_v}{\abs{\pi_{n+m}(G)}}\cdot |\ker \varphi_V^{n,m}|.
\end{align}
Hence, by comparing \textcolor{teal}{Equations (}\ref{align: rhs prob ind sections}\textcolor{teal}{)} and \textcolor{teal}{(}\ref{equation: horizontal vertical independence}\textcolor{teal}{)}, it is enough to prove that
$$|\ker \varphi_V^{n,m}|=\frac{|\pi_{n+m}(G)|}{|\pi_m(G)|^{\#V}\cdot |\pi_n(G)|}.$$
This follows from the first isomorphism theorem and the computation of the image of $\varphi_V^{n,m}$ above:
\begin{align*}
    |\ker \varphi_V^{n,m}|&=\frac{|\pi_{n+m}(\mathrm{St}_G(n))|}{|\varphi_V^{n,m}(\pi_{n+m}(\mathrm{St}_G(n)))|}=\frac{|\pi_{n+m}(\mathrm{St}_G(n))|}{|\pi_{m}(G)|^{\# V}}=\frac{|\pi_{n+m}(G)|}{|\pi_m(G)|^{\#V}\cdot |\pi_n(G)|}.\qedhere
\end{align*}
\end{proof}

Assuming that the group is branch and super strongly fractal, we enjoy of certain horizontal and vertical independence in the action of the group on the tree. The vertical independence appears because the action on the first $n$ levels is independent of the action of the sections at the vertices in $V$. This can be seen because in the numerator of \cref{equation: horizontal vertical independence}, the term $\abs{\pi_n(G)}$ is independent of the term $\prod_{v \in V} \# A_v$. The horizontal independence appears because the different actions below the vertices $V$ are independent, as in the product $\prod_{v \in V} \# A_v$ the terms are not related.

\subsection{Randomly evaluated words}

Now we are in position to prove \cref{Theorem: main result}, namely that a $k$-generated random subgroup of a super strongly fractal branch group is a free group of rank $k$ acting freely on $\partial T$ almost surely. The proof is similar to Abért's and Virág's proof in \cite[Proposition 4.1]{AbertVirag}, but with a few more subtleties, as we only have probabilistic independence of sections at vertices which are far enough away from each other. 

\begin{proof}[Proof of \cref{Theorem: main result}]

First, as $G$ is branch, by \cref{theorem: abert random free subgroup branch group} we have that $k$ Haar-random elements of $G$ generate a free subgroup $F_k$ of rank $k$ almost surely. Therefore, we just need to prove that the action of this subgroup on $\partial T$ is free almost surely.

For every $k \ge 1$ and $1\ne w\in F_k$, i.e. for $w$ a non-trivial word in $k$ letters, let us define the set
$$A_w := \set{(g_1,\dots,g_k) \in G^k: w(g_1,\dots,g_k) \text{ fixes finitely many vertices in $T$}},$$
where we view each word $w$ as the corresponding word map $w:G^k\to G$.

As there is only countably many words, it is enough to show that for every $k \ge 1$ and every $1\ne w\in F_k$, the set $A_w$ has measure $1$ in $G^k$, where $G^k$ is equipped with the normalized product measure $\mu^k$ induced by the Haar measure $\mu$ in $G$. We do this by induction on the length of $w$. The base case $|w|=1$ follows from \cref{theorem: ssf FPP 0}. Indeed, in this case $w$ is either $x_1$ or $x_1^{-1}$, and then 
\begin{align*}
\mu^1(A_w)&=\mu(A_w)=1-\mu(\text{elements in }G\text{ fixing an end in }\partial T)\\
&=1-\mathrm{FPP}(G)=1,
\end{align*}
where the last equality follows from \cref{theorem: ssf FPP 0} as $G$ is super strongly fractal.

Let us consider $w\in F_k$ such that $|w|=\ell$. By induction on the word length, we get $\mu^{k}(A_{\widetilde{w}})=1$ for every $\widetilde{w}\in F_k$ satisfying that $|\widetilde{w}|\le \ell-1$. Let us write $w = x_1 \dotsb x_\ell$. For every $1\le i\le \ell-1$, we further write $w_i = x_1 \dotsb x_i\in F_k$. Clearly, we have $|w_i|\le \ell-1$, so by the induction hypothesis $\mu^k(A_{w_i})=1$. In particular, given a random tuple $(g_1,\dotsc,g_k)$ there exists some level $N\ge 1$ such that for every $1\le i\le \ell-1$ the evaluation $w_i(g_1,\dots,g_k)$ moves every vertex at every level $n\ge N$.

Let us show that $w(g_1,\dotsc,g_k)$ does not fix an end in $\partial T$ almost surely. For that, it is enough to show that all its sections at level $N+D-1$ do not fix an end in $\partial T$ almost surely. In fact, if $w(g_1,\dotsc,g_k)$ fixes an end in $\partial T$ with positive probability, it must fix a vertex at level $N+D-1$ such that its section at that vertex fixes an end in $\partial T$ with positive probability.

Let $v\in \mathcal{L}_{N+D-1}$. Now, for any  $1\le i\le \ell-1$, we define $v_i\in T$ by 
$$v_i:=v\cdot {w_i(g_1,\dots,g_k)}.$$
Then, by \cref{align: property of sections} we get
$$w(g_1,\dotsc,g_k)|_v=x_1|_v\cdot x_2|_{v_1}\dotsb x_\ell|_{v_{\ell-1}}.$$
Furthermore, by \cref{lemma: moving a vertex and its descendant}, the vertices $v_1,\dotsc,v_{\ell-1}$ are not pairwise $(D-1)$-cousins, so the elements $x_1|_{v},x_2|_{v_1},\dotsc, x_\ell|_{v_{\ell-1}}$ are independent Haar-random elements by \cref{lemma: sections are independent}. Hence $w(g_1,\dotsc,g_k)|_v$ is a Haar-random element and thus it does not fix an end in $\partial T$ almost surely by \cref{theorem: ssf FPP 0}.
\end{proof}

\subsection{Fixed-point proportion of random subgroups}

We conclude the section by proving \cref{Theorem: FPP of random subgroups} and \cref{Corollary: FPP of random subgroups}.

Recall that if $G$ is a finite group, a \textit{section} of $G$ is a quotient $H/N$ where $H$ is a subgroup of $G$ and $N$ is a normal subgroup of $H$. A profinite group is said to be \textit{topologically finitely generated} if it contains a finitely generated dense subgroup. The following result is proved in \cite{MaximalGrowth} (note that there is a typo in the statement as it should be $2-\zeta(s)$ instead of $1-\zeta(s)$):

\begin{theorem}[{see {\cite[Theorem 1]{JorgeSantiFPP}}}]
For every finite group $F$ there is a constant $c(F)$ such that if $G$ is a finite $r$-generated group which does not have
sections isomorphic to $F$, then for all $s > 1$ such that
$k = c(F)r + s$ is an integer, then $k$ independent random elements of $G$ generate $G$ with probability greater than $2 - \zeta(s)$, where $\zeta$ is the Riemann zeta function.
\label{theorem: random elements profinite tfg}
\end{theorem}

We recall the statement of \cref{Theorem: FPP of random subgroups}: Let $T$ be the $d$-regular rooted tree and $G \le \Aut(T)$ a topologically finitely generated group such that $\mathrm{FPP}(G)=0$. Then, a $k$-generated random subgroup of $G$ has null fixed-point proportion with probability tending to 1 as $k \to \infty$.

\begin{proof}[Proof of \cref{Theorem: FPP of random subgroups}]
As $G$ is topologically finitely generated, there exists a number $r \geq 1$ such that $\pi_n(G)$ is $r$-generated for all $n \geq 1$. As $\pi_n(G) \leq \Aut(T^n)$ for any $n\ge 1$, none of the subgroups $\pi_n(G)$ can have elements of prime order $p>d$. Therefore, no section of $\pi_n(G)$ can be isomorphic to $C_p$, where $C_p$ denotes a cyclic group of order $p$. Hence, by \cref{theorem: random elements profinite tfg}, we get that $k$ independent Haar-random elements of $\pi_n(G)$ generate $\pi_n(G)$ with probability greater than $2 - \zeta(k-c(C_p)r)$ for every $n\ge 1$. Note that the lower bound $2 - \zeta(k-c(C_p)r)$ does not depend on~$n$. Since $\pi_n$ is measure-preserving, we conclude that $k$ independent Haar-random elements generate~$G$ (topologically) with probability greater than $2 - \zeta(k-c(C_p)r)$.

If $H_k$ is a subgroup of $G$ generated by $k$ independent Haar-random elements, then 
$$\mathrm{FPP}(H_k) = \mathrm{FPP}(\overline{H_k}) = \mathrm{FPP}(G) = 0$$ 
with probability greater than $2 - \zeta(k-c(C_p)r)$. As $k$ goes to infinity, we have $k-c(C_p)r \rightarrow +\infty$ and thus $\zeta(k-c(C_p)r) \rightarrow 1$, giving the result.
\end{proof}

Let $W_p$ be the group of $p$-adic automorphisms, i.e. the iterated wreath product corresponding to the subgroup generated by a $p$-cycle $\sigma\in \mathrm{Sym}(p)$.

Recall now the statement of \cref{Corollary: FPP of random subgroups}: If $G \le W_p$ is a topologically finitely generated pro-$p$ group with $d$ generators such that $\mathrm{FPP}(G)=0$, then, for $k \geq d$, a $k$-generated random subgroup of $G$ has null fixed-point proportion with probability at least 
$$\prod_{j = 0}^{d-1} \left(1 - \frac{1}{p^{k-j}} \right).$$
In particular this probability tends to 1 as $k \to \infty$.

\begin{proof}[Proof of \cref{Corollary: FPP of random subgroups}]
    For a pro-$p$ group $G\le W_p$, its maximal subgroups are in one-to-one correspondence with the maximal subgroups of its Frattini quotient $G/\Phi(G) \cong \mathbb{F}_p^d$. A $k$-random subgroup generates $G$ topologically if and only if its closure is not contained in any maximal subgroup of $G$. By looking at $G/\Phi(G)$, the probability of this happening is precisely the probability of $k$-random independent vectors in $\mathbb{F}_p^d$ generating the full space~$\mathbb{F}_p^d$. Let us denote this probability by $P_p^d(k)$. It is clear that $P_p^d(k) > 0$ if and only if $k \geq d$, as we need at least $d$ vectors to generate $\mathbb{F}_p^d$. The problem of having $k$-random independent vectors in $\mathbb{F}_p^d$ generating~$\mathbb{F}_p^d$ can be restated as the probability of a random $k \times d$-matrix $A$ with coefficients in~$\FF_p$ to have rank $d$. As $\mathrm{rank}(A) = \mathrm{rank}(A^t)$, this is equivalent to finding $d$-random linearly independent vectors in $\mathbb{F}_p^k$. Therefore $$P_p^d(k) =\frac{\prod_{j = 0}^{d-1} (p^k-p^j)}{p^{kd}} = \prod_{j = 0}^{d-1} \left(1 - \frac{1}{p^{k-j}} \right).$$

As in \cref{Theorem: FPP of random subgroups}, we have that the subgroup $H_k$ generated by $k$ Haar-random elements generates $G$ topologically with probability $P_p^d(k)$, in which case we get
$$\mathrm{FPP}(H_k) = \mathrm{FPP}(\overline{H}_k) = \mathrm{FPP}(G) = 0.$$
Finally, note that $P_p^d(k) \rightarrow 1$ as $k \rightarrow \infty$.
\end{proof}

\section{Counterexamples}
\label{section: examples}

In this section, we show that \cref{Theorem: main result} does not hold if one drops either the assumption on the super strong fractality or on the branchness of the group.

\subsection{Super strongly fractal groups}

Let $A\le W_p$ be the super strongly fractal group in \cite[Section 3]{JorgeCyclicity}, i.e. a realization of the elementary $p$-abelian group $ \prod_{n\ge 1} C_p$ as a super strongly fractal subgroup of $W_p$. The group~$A$ is no more than the group $G_\mathcal{S}$ with defining sequence $\mathcal{S}:=\{D_{p^n}(\sigma)\}_{n\ge 0}$; see \cite[Section 4]{JorgeSpectra}.

As $A$ is elementary $p$-abelian, any finitely generated subgroup is finite. This shows that a $k$-generated random subgroup $H_k$ of $A$ is isomorphic to $C_p\times\overset{k}{\ldots}\times C_p$ almost surely. Thus, the statement in \cref{Theorem: main result} does not hold for $A$. What is more, as $H_k$ is finite, its fixed point proportion is positive. Indeed
$$\mathrm{FPP}(H_k)=|H_k|^{-1}>0$$ 
as $A$ acts freely on the boundary of the $p$-adic tree. However, since $A$ is super strongly fractal it has null fixed-point proportion by \cref{theorem: ssf FPP 0}. Thus, this answers \cite[Question~7.3]{Santi} on the negative: there exists a profinite group with null fixed-point proportion such that a $k$-generated random subgroup has positive fixed-point proportion with probability~1.

The same construction above works for any $d$-regular rooted tree. Indeed, simply choose $\sigma$ a $d$-cycle and $\mathcal{S}:=\{D_{d^n}(\sigma)\}_{n\ge 0}$ and consider the group $G_\mathcal{S}$ instead.

\subsection{Groups of finite type}

Given $T$ the $d$-regular rooted tree, recall that we can see the vertices of $T$ as finite words over the alphabet $X := \set{1, \dots, d}$. Identifying the alphabet $X$ with $\mathbb{Z}/d\mathbb{Z}$, consider $\mathcal{Q}\trianglelefteq \mathcal{P}\le \mathrm{Sym}(d)$ defined via
$$\mathcal{P} := \set{x \mapsto ax+b: a \in (\mathbb{Z}/d\mathbb{Z})^\times, b \in \mathbb{Z}/d\mathbb{Z}},$$
and
$$\mathcal{Q} := \set{x \mapsto x+b: b \in \mathbb{Z}/d\mathbb{Z}}.$$
Define the group $G_\mathcal{Q}^\mathcal{P} \leq \Aut(T)$ as $$G_\mathcal{Q}^\mathcal{P} := \set{g \in \mathrm{Aut}(T): g|_v^1 \in \mathcal{P} \text{ and } (g|_v^1)(g|_w^1)^{-1} \in \mathcal{Q} \text{ for all $v,w\in T$}}.$$

It was proved by the second author in \cite[Chapter 6]{Santi} that the group $G_\mathcal{Q}^\mathcal{P}$ corresponds to the iterated Galois group of the polynomial $x^d+1$ over $\mathbb{Q}$; for the definition of an iterated Galois group, see for instance \cite[Section 2.5]{Santi}.

The group $G_\mathcal{Q}^\mathcal{P}$ is not super strongly fractal as the labels of $\mathrm{St}_{G_\mathcal{Q}^\mathcal{P}}(n)$ all lie in~$\mathcal{Q}$ for any $n\ge 1$ but $G_\mathcal{Q}^\mathcal{P}$ contains elements with labels in $\mathcal{P}\setminus \mathcal{Q}$, so $\varphi_v(\mathrm{St}_{G_\mathcal{Q}^\mathcal{P}}(n))$ is strictly contained in $G_\mathcal{Q}^\mathcal{P}$ for any $n\ge 1$ and $v\in \mathcal{L}_n$. However, the group $G_\mathcal{Q}^\mathcal{P}$ is a group of finite type by \cite[Proposition 4.2]{Santi}. In particular, $G_\mathcal{Q}^\mathcal{P}$ is closed and regular branch. Moreover, it is also proved in \cite[Section 5.1]{Santi} that $G_\mathcal{Q}^\mathcal{P}$ has positive fixed-point proportion for $d$ odd. This implies that a random subgroup does not act freely on the boundary $\partial T$ with probability $1$.

\bibliographystyle{unsrt}

\end{document}